\documentclass{amsart}
\usepackage{amsmath}
\usepackage{graphicx}
\usepackage{latexsym}
\usepackage[all]{xy}

\newtheorem{thm}{Theorem}

\newtheorem{cor}[thm]{Corollary}
\newtheorem{prop}[thm]{Proposition}

\theoremstyle{definition}

\newtheorem{rk}[thm]{Remark}
\newtheorem{remark}[thm]{Remark}

\newcommand{\R}{\mathbb{R}}
\newcommand{\Z}{\mathbb{Z}}

\def \Push {\mathcal Push}
\def \x {\times}
\def \eu{{\text{e}}}

\DeclareMathOperator{\Diff}{Diff}

\DeclareMathOperator{\Map}{Map}
\DeclareMathOperator{\Homeo}{Homeo}

\DeclareMathOperator{\Hyper}{\mathcal{H}^2}

\begin{document}

\title[Flat bundles and commutator lengths]
{Flat bundles and commutator lengths}

\author[R. \.{I}. Baykur]{R. \.{I}nan\c{c} Baykur}
\address{Max Planck Institute for Mathematics, Bonn, Germany \newline
\indent Department of Mathematics, Brandeis University, Waltham MA, USA}
\email{baykur@mpim-bonn.mpg.de, baykur@brandeis.edu}

\begin{abstract}
The purpose of this article is two-fold: We first give a more elementary proof of a recent theorem of Korkmaz, Monden, and the author, which states that the commutator length of the $n$-th power of a Dehn twist along a boundary parallel curve on a surface with boundary $\Sigma$ of genus $g \geq 2$ is $\lfloor \frac{|n|+3)}{2} \rfloor$ in the mapping class group $\Map(\Sigma)$. The alternative proof we provide goes through push maps and Morita's use of Milnor-Wood inequalities, in particular it does not appeal to gauge theory. In turn, we produce infinite families of pairwise non-homotopic $4$-manifolds admitting genus $g$ surface bundles over genus $h$ surfaces with distinguished sections which are flat but admit no flat connections for which the sections are flat, for every fixed pairs of integers $g, h \geq 2$. The latter result generalizes a theorem of Bestvina, Church, and Souto, and allows us to obtain a simple proof of Morita's non-lifting theorem (for an infinite family of non-conjugate subgroups) in the case of marked surfaces.
\end{abstract}

\maketitle

\setcounter{secnumdepth}{2}
\setcounter{section}{0}


\section{Introduction}

The study of commutator lengths in various structure groups for fiber bundles has a long history in topology, where Milnor's 1958 paper on flat disk bundles \cite{M} and Wood's follow-up 1971 paper on flat circle bundles \cite{W} played avant-garde roles. More recently, several results on (stable) commutator lengths in mapping class groups of surfaces were obtained with the help of gauge theory; see for instance \cite{EK, Kork, kot}. In \cite{BKM}, Korkmaz, Monden, and the author proved the following theorem in the same vein: Let $\delta$ be a boundary parallel simple-closed curve on an orientable surface $\Sigma$ of genus $g \geq 2$ with boundary, and let $t_{\delta}$ denote the positive Dehn twist along $\delta$ in the mapping class group $\Map(\Sigma)$. Then the commutator length of $t_{\delta}^n$ is $\lfloor \frac{|n|+3}{2} \rfloor$, the floor of $\frac{|n|+3}{2} \in \Z[\frac{1}{2}]$. This led the first precise calculation of a non-zero stable commutator length of an element in a mapping class group of a surface of genus $g \geq 2$. (See Theorem~\ref{BKMtheorem} below.) The authors' proof of this theorem relied on celebrated results on Seiberg-Witten invariants of symplectic $4$-manifolds, and mapping class group factorizations featuring a generalized lantern relation. We will give a more elementary proof of this theorem by reviewing an argument of Morita in \cite{Morita} using euler classes of fiber bundles and Milnor-Wood inequalities on one hand, and by employing push maps on the other; see Section~\ref{sections}. 

A curious open question on surface bundles asks whether or not there exists a closed surface bundle over a closed surface which does not admit a flat connection. An `approximate' answer to this problem was given by Bestvina, Church, and Souto, who proved that the Atiyah-Kodaira surface bundles do not admit flat connections for which some distinguished sections are flat \cite{BCS}. The Atiyah-Kodaira examples are holomorphic bundles on complex surfaces of general type, whose construction for given fiber and base genera $(g,h)$ (with the desired sections of high self-intersection numbers) is a rather challenging task. Moreover, Parshin's proof of the Geometric Shafarevich Conjecture implies that there can be at most finitely many of such examples for fixed $(g,h)$. However, the holomorphicity of these bundles is irrelevant to the above question --- as we will see that it can be dropped so as to obtain a much more general result, which is the main theorem of this article:

\begin{thm} \label{mainthm}
For every fixed pair of integers $g, h \geq 2$ such that $h \geq 4$ for $g=2$, there is an infinite family of pairwise non-homotopic $4$-manifolds, each of which admits a genus $g$ surface bundle over a genus $h$ surface with a distinguished section, such that the bundle admits a flat connection but does not admit any for which the section is parallel. 
\end{thm}

\noindent As suggested by the abundance of the bundles we build versus the finiteness result of Parshin's, we will note that the total spaces of these bundles do not admit complex structures with either orientations (although all can be made symplectic), which in this sense complements the examples of \cite{BCS}. 

In \cite{Morita}, Morita showed that the exact sequence
\[
1 \to \Diff ^+_0(\Sigma) \to \Diff ^+ (\Sigma) \to \Map(\Sigma) \to 1
\]
is not right-split for $\Sigma$ an orientable genus $g \geq 2$ surface with $k \geq 1$ marked points. Here $\Diff ^+_0(\Sigma)$ denotes the normal subgroup of $\Diff ^+ (\Sigma)$ which consists of elements that are isotopic to identity via isotopies fixing the marked points. We give a new proof of this theorem, which follows from:

\begin{cor} \label{cormorita}
For each $g, h \geq 2$, there are homomorphisms $\pi_1(\Sigma_{h}) \to \Map (\Sigma)$, where $\Sigma$ is a genus $g$ surface with $r \geq 1$ marked points, whose images do not lift to $\Diff ^+ (\Sigma)$. There are infinitely many conjugacy classes of such mappings for each $g, h \geq 2$, provided $h \geq 4$ if $g=2$. 
\end{cor}

\noindent The strategy for deriving the above corollary from Theorem~\ref{mainthm} is similar to that of \cite{BCS}, except that our topological constructions drastically simplify the generalization of the $r=1$ case to any $r>1$.  

\vspace{0.2in}
\noindent \textit{Acknowledgements.} The author would like to thank Dan Margalit and Ursula Hamenst\"adt for their comments on our earlier work in \cite{BKM}, which led the author to explore the more elementary arguments given for the main result of \cite{BKM} on the commutator length of Dehn twists along boundary parallel curves. We are also grateful to the anonymous referee for pointing out the connections to Morita's work in \cite{Morita2}. The author was partially supported by the NSF grant DMS-0906912.

\vspace{0.2in}
\section{Sections of flat bundles and commutator lengths} \label{sections}

Let $\Sigma_{g,r}^s$ denote a compact oriented surface of genus $g$ with $s$ boundary components and $r$ marked points in the interior. We denote by $\Diff ^+ (\Sigma)$ the group of orientation preserving self-diffeomorphisms of $\Sigma=\Sigma_{g,r}^s$ which restrict to the identity on the marked points and on some open neighborhood of the boundary. The identity component of this group is denoted by $\Diff_0 ^+ (\Sigma)$. The \emph{mapping class group\,} of $\Sigma$ is then defined as $\Map(\Sigma)=\Diff^+(\Sigma)/ \Diff_0^+(\Sigma)$. Whenever $r$ or $s$ is zero, we simply drop it from the notation.

Let $f: X \to B$ be a (locally trivial) oriented bundle over a compact oriented surface $B$ with a regular fiber $F \cong \Sigma_g^s$ of genus $g \geq 1$. Note that when $s >0$, the restriction of $f$ to $\partial X$ is a circle bundle over $B$. Let $\sigma: B \to \text{Int}(X)$ be a section of $f: X \to B$, then the $2$-dimensional submanifold $S=\sigma(B)$ of $X$ intersects each fiber positively at one point. Conversely, whenever we have such a submanifold $S \subset \text{Int}(X)$ intersecting each fiber at one point, we can define a section $\sigma$ with image $S$. So by a slight abuse of language, we will also call $S$ a \emph{section} of $f$. 

A theorem of Earle and Eells, and its extension to compact surfaces with boundary by Earle and Schatz, show that $\text{Diff}^+_0(\Sigma_{g}^s)$ is contractible for $g \geq 2$. It follows that the bundles $f: X \to B$ with regular fiber $F \cong \Sigma_g$ and with distinguished sections $S_f$ are classified (up to bundle isomorphisms) by the \textit{monodromy representations} $\mu_f: \pi_1(B) \to \Map(\Sigma_{g,1})$ (up to conjugacy), and such a bundle admits a flat connection for which $S$ is parallel if and only if the corresponding representation can be lifted to a map $\pi_1(B) \to \text{Diff}^+_0(\Sigma_{g,1})$. The analogous statements hold when we drop the assumption on the sections and the marked point above. 

\subsection{Upper bounds on the self-intersection numbers of sections} \

A sharp upper bound on the self-intersection number of a section of a surface bundle over a surface was originally obtained by Morita in \cite{Morita2}, and was later obtained using gauge theory in \cite{BKM} and \cite{Bow}, respectively. This bound constitutes one side of the inequality leading to the calculation of the precise commutator length of the boundary parallel Dehn twist in $\Map(\Sigma_g^1)$ \cite{BKM}. For completeness, we provide a detailed account of the arguments yielding to the proof of this fact, which  follow the same approach as in \cite{Morita2}.

For $g \geq 2$, a classical construction due to Nielsen shows that the self-diffeo-morphisms of $\Sigma_g$ fixing a marked point $p \in \Sigma_g$ can be lifted to $\pi_1(\Sigma_g)$-equivariant quasi-isomorphisms of the hyperbolic space $\Hyper$, which in turn can be extended to obtain a well-defined homeomorphism on the boundary $\partial \Hyper \cong S^1$. Using the fact that $\Sigma_g$ is a $K(\pi_1(\Sigma_g), 1)$, it is then easy to see that we get a monomorphism (e.g. \cite[Section~5.5.6]{FM})
\[
\Map(\Sigma_{g,1}) \to \Homeo ^+(S^1) \, .
\]
It then follows from the classical result of Wood's \cite{W} that the euler number of this circle bundle on $S$ in absolute value is bounded above by $2h-2$ for $h \geq 1$. 

If $f:X \to \Sigma_h$ admits a flat connection for which $S$ is parallel, then we get the following lift:
\[
\xymatrix{ & \Diff ^+(\Sigma_{g,1}) \ar[d] \\\ \pi_1(B) \ar[r]^{\mu _f} \ar@{-->}[ur]^{\bar{\mu} _f} & \Map (\Sigma_{g,1}) \, } 
\]
of the monodromy map $\mu_f$. In this case we obtain a refinement of the above observation: We have an induced flat structure on the normal bundle $\nu_S$ of $S$, which is given by composing $\bar{\mu} _f$ with the derivative map at the fixed point $p$:
\[\text{Diff}^+(\Sigma_{g,1}) \stackrel{D_p} \rightarrow GL^+(T_p \Sigma_h) = GL^+(2, \R).\]
As shown by Milnor in \cite{M}, in this case, the euler number of the normal disk bundle on $S$ (and thus that of the circle bundle) in absolute value is bounded above by $h-1$ instead, for $h \geq 1$. 

Since the circle bundles we get on $S$ are boundaries of normal (fibered) disk bundles on it (e.g. \cite[Proof of Theorem~1.2]{BCS}), we conclude that the euler number of these bundles is nothing but the self-intersection number $[S]^2$. We will revisit this claim in the next subsection; for now we can summarize the above discussion as 

\begin{prop} \label{labelprop}
If $S$ is a section of a genus $g$ surface bundle $f: X \to \Sigma_h$ with $g, h \geq 1$, then $|[S]^2| \leq 2h-2$. If $f$ admits a flat connection for which $S$ is parallel, then this bound improves to $|[S]^2| \leq h-1$. 
\end{prop}


\begin{rk}
In \cite{BKM} it was moreover shown that a similar upper bound holds for Lefschetz fibrations if we remove the absolute value; a proof of which, \textit{for sections that miss the critical locus}, can be obtained using the above approach and following the framework of \cite{Smith} for example.) The proof in \cite{BKM} as well as the one in \cite{Bow} uses gauge theory, appealing to Seiberg-Witten invariants on minimal symplectic $4$-manifolds and the adjunction inequality for Seiberg-Witten basic classes. As seen from the above discussion however, the use of gauge theory is rather superfluous.
\end{rk}

\subsection{Surface bundles and the euler class via mapping class groups} \

Once again, assume that $g \geq 2$, and recall the Nielsen construction we sketched above. Define $\widetilde{\Map (\Sigma_g, 1)}$ as the pull-back of $\Map (\Sigma_{g,1})$ to $\widetilde{\Homeo}^+(S^1)$, the subgroup of $\Homeo ^+(\R)$ consisting of homeomorphisms commuting with translation by $1$. Then the central extension
\[
1 \to \Z \to \widetilde{\Homeo} ^+(S^1) \to \Homeo ^+ (S^1) \to 1 
\]
gives rise to the central extension
\[
1 \to \Z \to \widetilde{\Map} (\Sigma_{g,1}) \to \Map (\Sigma_{g,1}) \to 1 .
\]
The cocycle $\eu \in H^2(\Map(\Sigma_{g,1}; \Z))$ associated to this central extension is called the \emph{euler class} for $\Map (\Sigma_{g,1})$, which is by construction the pullback of the euler class $\eu_{S^1}$ of $\Homeo ^+(S^1)$ under the inclusion-induced homomorphism 
\[ H^2(\Map(\Sigma_{g,1}; \Z)) \to H^2(\Homeo ^+(S^1); \Z) \, . \]

On the other hand, we have the following central extension obtained from the boundary capping homomorphism
\[
1 \to \Z \to \Map(\Sigma^1_g) \to \Map(\Sigma_{g,1}) \to 1 \, ,
\]
where the kernel $\Z$ is generated by Dehn twists along a boundary parallel curve $\delta$ on $\Sigma^1_g$. It is well-known that the euler class of $\Map(\Sigma_{g,1})$ associated from this last central extension agrees with the former (see for instance \cite[Section~5.5.6]{FM}). As we will see shortly, the evaluation of the euler class $\eu$ on a $2$-class in $H_2(\Map{\Sigma_{g,1}}; \Z) \cong \Z$ will give an integer, which is precisely the power $n$ of the boundary parallel Dehn twist $t_{\delta}$. 

The monodromy of a surface bundle $f\colon X \to B$ with regular fiber $F \cong \Sigma_g$ and base $B \cong \Sigma_h$, is given by a factorization of the identity element $1\in \Map(\Sigma_g)$ as
\begin{equation} \label{monodromypresentation}
1=\prod _{i=1}^h [\alpha_i, \beta_i] \, ,
\end{equation}
where $\alpha_i$ and $\beta_i$ are images of standard basis elements $a_i, b_i$ of $\pi_1(B)$ under the monodromy map $\mu_f: \pi_1(B) \to \Map(\Sigma_g)$. The choice of a base point amounts to determining this factorization only up to global conjugation. 

Now, a section of $\sigma: B \to X$ gives rise to a lift of the relation (\ref{monodromypresentation}) to a relation of the same form in $\Map(\Sigma_{g,1})$. Moreover, if we remove a fibered tubular neighborhood $N$ of the section $S=\sigma(B)$ from $X$, we obtain a surface bundle $f|_{X \setminus N} \colon X \setminus N \to B$ with regular fibers homeomorphic to $\Sigma_g^1$, which prescribes a further lift of this relation in $\Map(\Sigma_{g,1})$ to a relation of the form:
\begin{equation} \label{liftedmonodromypresentation}
t_{\delta}^{n}=\prod _{i=1}^h [\tilde{\alpha_i}, \tilde{\beta_i}] \, ,
\end{equation}
where $\tilde{\alpha_i}, \tilde{\beta_i}$ are the lifts of $\alpha_i, \beta_i$, and $t_{\delta}$ is a positive Dehn twist along the boundary parallel curve $\delta$ in $\Sigma_g^1$. 

Considering the unit normal disk bundle on $S$ obtained by the derivative map along $\sigma$, we get a fibered tubular neighborhood $N$ of $S$ isomorphic to $\nu S$. This normal neighborhood is clearly isomorphic to the one we get by capping off every fiber component diffeomorphic to $\Sigma_g^1$ by a $2$-disk, where the centers of these disks trace the section $S$. Choosing a framing on the $2$-disk that caps off $\Sigma_g^1$ amounts to choosing a marked point on its boundary, and in turn, we obtain a push-off $S'$ of $S$ on the boundary of its tubular neighborhood. We can then see that the negative of the power $n$ of the boundary parallel Dehn twist $t_{\delta}$ above equals to the intersection number of $S'$ and $S$ \cite{OS}, and therefore is equal to the self-intersection number $[S]^2$, as claimed in the previous subsection. 

Conversely, whenever we have a factorization as in (\ref{liftedmonodromypresentation}) above, we can construct a surface bundle $f: X \to B$, with a regular fiber $F \cong \Sigma_g$ and base $B \cong \Sigma_h$, together with a distinguished section $S$ of self-intersection $-n$. Note that all the arguments above can be generalized to surface bundles which have other boundary components as well. 

Hence, by Proposition~\ref{labelprop}, we see that $t_\delta^{n}$ can be written as at most $h=\lfloor (n+3)/2 \rfloor $ commutators in $\Gamma_g^1$, since such an expression would correspond to a genus $g$ surface bundle over a genus $h$ surface with a section $S$ of self-intersection number $-n$. It follows that the stable commutator length of $t_\delta$ is bounded below by $1/2$.

\subsection{Realizing the upper bound and the commutator length calculation} \

We will now show that the lower bound $1/2$ for the stable commutator length of a boundary parallel Dehn twist $t_\delta$ is indeed achieved by expressing $t_\delta^k$ as the maximum number of commutators allowed by Proposition~\ref{labelprop}. 

In \cite{BKM}, Korkmaz, Monden, and the author obtained the following relation in $\Map (D)$, where $D$ is the $2$-sphere with five holes, and $\delta$, $a_1, \ldots, a_4$, $x_1, \ldots, x_4$ are as shown in Figure~\ref{lantern}.
\begin{eqnarray}\label{eqn}
t_\delta^{2k} 
 &=&\left( \prod_{i=1}^{k}  (t_{x_1}t_{a_2}^{-1} t_{x_2} t_{a_1}^{-1})^{t_{x_3}^{i-1}} \right)
 (t_{x_3}^k t_{a_4}^{-k} t_{x_4}^k t_{a_3}^{-k} ) \, .
\end{eqnarray} 
Here $\beta ^{\gamma}$ is a shorthand notation for the conjugation $\gamma \beta \gamma^{-1}$, for any $\beta, \gamma$ in the same mapping class group. We will reprove this fact by using ``push maps'', which we will describe shortly.

\begin{figure}[ht]
 \begin{center}
      \includegraphics[width=7cm]{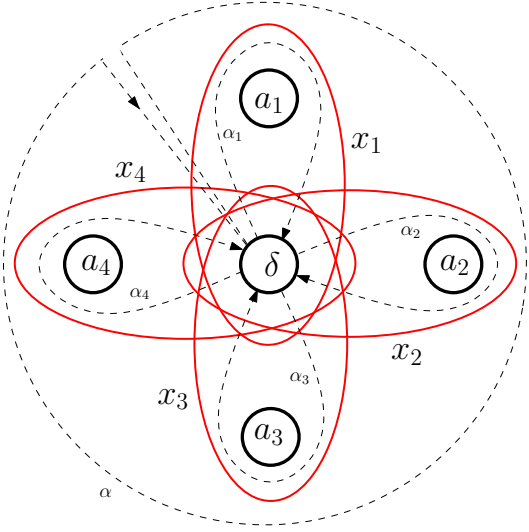}
      \caption{The $2$-sphere $D$ with five boundary components $\delta, a_1, a_2, a_3, a_4$. The oriented paths $\alpha, \alpha_1, \alpha_2, \alpha_3, \alpha_4$ for the push maps are as shown.}
      \label{lantern}
   \end{center}
\end{figure}

There is an elementary self-diffeomorphism of $D$, which is obtained by dragging $\delta$ around other boundary components along an oriented embedded arc $\alpha$ without turning any one of the boundary circles. This is called the \emph{push map of $\delta$ along $\alpha$}, which we denote by $\Push_{\delta}(\alpha)$. Let $A$ be an annulus with one hole (i.e. a pair of pants), where the boundary of its hole is $\delta$, and where $\alpha$, along with $\delta$, is enclosed between its other two boundary components. Using the orientation on $\alpha$ and $D$ we can label the two boundary curves of $A$ by $\alpha'$ and $\alpha''$ as the one ``to the left'' of $\alpha$ and the one ``to the right'' of it, respectively. A key observation is that $\Push_{\delta}(\alpha)$ can be expressed as
\[ 
\Push_{\delta}(\alpha) = t_{\alpha'} t_{\delta}^{-1} t_{\alpha''}^{-1}  
\]
in $\Map (D)$. (None of these are specific to $D$ or $\delta$ of course; the reader can turn to \cite[Section~4.2.1]{FM} for a general treatment.) 

We will use push maps to derive the following generalized lantern relation obtained in \cite{BKM}:
\begin{equation} \label{generalizedlantern}
t_\delta^{2} \,t_{a_1}t_{a_2} t_{a_3} t_{a_4} = t_{x_1} t_{x_2} t_{x_3} t_{x_4} \, ,
\end{equation}
which is equivalent to the \textit{daisy relation} of \cite{EMV}. We believe that the argument we will present using push maps is very easy to visualize and can be directly generalized to derive the analogous relation in a $2$-sphere with $n+4$ holes for any $n \geq 0$ ---with $n=0$ giving the famous lantern relation. 

Let $\alpha_i$ be embedded arcs in $D$ which enclose $a_i$ and have both of its end points on the boundary component $\delta$, for $i=1, \ldots, 4$, and $\alpha$ be the arc that encloses all $a_i$, as shown in Figure~\ref{lantern}. We orient all of them clockwise, and view them as ``loops based at $\delta$ \, '' ---so the end points of the arcs on $\delta$ do not matter. From the naturality of the push map it follows that 
\[ 
\Push_{\delta}(\alpha_1) \, \Push_{\delta}(\alpha_2) \, \Push_{\delta}(\alpha_3) \, \Push_{\delta}(\alpha_4) = \Push_{\delta}(\alpha_1 \alpha_2 \alpha_3 \alpha_4) = \Push_{\delta}(\alpha) \, 
\]
(where the concatenation of the loops goes from left to right), which translates to
\[ 
t_{x_1} t_{\delta}^{-1} t_{a_1}^{-1} t_{x_2} t_{\delta}^{-1} t_{a_2}^{-1}  t_{x_3} t_{\delta}^{-1} t_{a_3}^{-1} t_{x_4} t_{\delta}^{-1} t_{a_4}^{-1} = t_{\delta}^{-1} t_{\delta}^{-1} \, , 
\]
since the outer boundary component of the annulus containing $\alpha$ is null-homotopic in $D$. As $a_i$ and $\delta$ are disjoint from all the other curves, $t_{a_i}$ and $t_{\delta}$ commute with all the other Dehn twists that appear above, allowing us to rewrite this equation to obtain the relation (\ref{generalizedlantern}). 

Taking the $k$-th power of both sides of (\ref{generalizedlantern}) we have
\begin{eqnarray*}
t_\delta^{2k} t_{a_1}^k t_{a_2}^k t_{a_3}^k t_{a_4}^k t_{x_4}^{-k}&=&(t_{x_1}t_{x_2}t_{x_3})^k\\
&=& (t_{x_1}t_{x_2}) (t_{x_1}t_{x_2})^{t_{x_3}}  (t_{x_1}t_{x_2})^{t_{x_3}^2} \cdots (t_{x_1}t_{x_2})^{t_{x_3}^{k-1}}
t_{x_3}^k\\
&=& \left( \prod_{i=1}^{k}  (t_{x_1}t_{x_2})^{t_{x_3}^{i-1}} \right)t_{x_3}^k \, .
\end{eqnarray*}
Lastly, we take the Dehn twists along $a_i$ on the left to the right-hand side to obtain the relation (\ref{eqn}). 
We can now embed $D$ into a genus $g \geq 2$ surface $\Sigma$ with two boundary components $\delta$ (which is the same as the $\delta$ boundary of $D$) and $\delta_0$ by identifying $a_1$ and $a_4$ as in Figure~\ref{relationcurves}. So we obtain the relation: 
\begin{eqnarray}\label{eqn:b}
t_\delta^{2k} 
 &=&\left( \prod_{i=1}^{k}  (t_{x_1}t_{a_2}^{-1} t_{x_2} t_{a_1}^{-1})^{t_{x_3}^{i-1}} \right)
 (t_{x_3}^k t_{a_4}^{-k} t_{x_4}^k t_{a_3}^{-k} ) \\
 &=&\left (\prod_{i=1}^{k}  [t_{x_1}t_{a_2}^{-1}, \phi ]^{t_{x_3}^{i-1}} \right)  [t_{x_3}^k t_{a_4}^{-k}, \psi]   \, .
\end{eqnarray}
in $\Map (\Sigma)$. Here $\phi$ and $\psi$ are orientation preserving self-diffeomorphisms of $\Sigma$ relative to $\partial \Sigma$ mapping $x_1, a_2$ to $a_1, x_2$ and $x_3, a_4$ to $a_3, x_4$ in the same order, respectively. For later purposes, we can moreover assume that both $\phi$ and $\psi$ are supported on the genus $2$ subdomain of $\Sigma$ containing all these curves and $\delta$ (and not $\delta_0$; see Figure \ref{relationcurves}), and such that they become isotopic to identity once the boundary component $\delta$ is capped off. Whenever there is such a self-diffeomorphism, we can insert a commutator; for example 
\[t_{x_3}^k t_{a_4}^{-k} t_{x_4}^k t_{a_3}^{-k} = t_{x_3}^k t_{a_4}^{-k} t_{\psi(a_4)}^k t_{\psi(x_3)}^{-k} \\
= t_{x_3}^k t_{a_4}^{-k} \psi \, t_{a_4}^{k} t_{x_3}^{-k} \psi^{-1} = [t_{x_3}^k t_{a_4}^{-k}, \psi ] . 
\]
Thus we see that $t_\delta^{2k}$ can be expressed as the product of $k+1$ commutators in $\Sigma$.

Writing $t_\delta^{2k}$ as a product of $k+1$ commutators as above and using the lantern relation, one can also show that
$t_\delta^{2k+1}$ can be expressed as a product of $k+2$ commutators; the argument for which is readily available in \cite{BKM}, which we will not repeat here. 

Hence, having realized the upper bound given in Proposition~\ref{labelprop}, we obtained a new proof of the following theorem:

\begin{thm}[Baykur, Korkmaz and Monden \cite{BKM}] \label{BKMtheorem}
Let $g\geq 2$, $n\geq 1$ and let $\Sigma$ be a compact connected oriented surface of genus $g$ with boundary. Let $\delta$ be one of the boundary components of $\Sigma$. Then the $n^{th}$ power $t_\delta^{n}$ of the Dehn twist about $\delta$ is a product of $\lfloor (n+3)/2 \rfloor $ commutators. It follows that the stable commutator length of $t_{\delta}$ is $1/2$. 
\end{thm}

\begin{remark}\label{refremark}
As pointed out by the anonymous referee, the stable commutator length calculation $scl(t_{\delta})=1/2$ which came as a direct corollary of our commutator length calculation above and the fact that the \textit{vertical euler class} $||e||=1/2$ calculated in \cite{Morita2} is not a mere coicidence. Here is a sketch of the argument which was generously provided by the referee: the image of the monodromy homomorphism $\pi_1(B) \to \Map(\Sigma_g)$ lifts to one into $\Map(\Sigma^1_g)$ with boundary on copies of $t_{\delta}$, and the degree with which the boundary winds around delta is precisely equal to the euler class of the closed surface group $\pi_1(B)$ in $\Map(\Sigma_g)$. So by interpreting the stable commutator length function $scl$ as a scaled $L^1$ pseudo-norm (as in \cite{Calegari}), one sees that $scl(t_{\delta})$ is $1/4$ times the $L^1$ norm of the integral class in $H_2(\Map(\Sigma_g))$ --- the factor of $1/4$ comes from the fact that $scl$ counts handles whereas the $L^1$ norm counts triangles. The $L^1$ norm of this integral class is $2$, thus giving us $scl = 1/2 = ||e||$. 
\end{remark}

\vspace{0.2in}
\section{Surface bundles and flat connections} \label{flat}

We begin with proving Theorem~\ref{mainthm}. From the relation~\ref{eqn} we already have a genus $g$ surface
bundle over a genus $h$ surface with a section of self-intersection $2h-2$. We will see below that this bundle is 
flat. Our first goal is to slightly modify this bundle to obtain an infinite family of bundles with the same fiber
and base genera, possessing a section of self-intersection $2h-2$. These will still be flat as surface bundles. Finally, by computing the first homology of the total spaces of these bundles, we will verify that we obtain pairwise non-homotopy equivalent bundles. 

\begin{proof}[Proof of Theorem~\ref{mainthm}.] 

Let $\Sigma$ be a genus $g \geq 2$ \, surface with two boundary components $\delta$ and $\delta_0$. As shown above, the relation (\ref{eqn}) holds in $\Map (\Sigma)$, which prescribes a genus $g$ surface bundle over a genus $h=k+1$ surface with two sections $S$ and $S_0$ of self-intersections $2h-2$ and $0$, respectively.

Building upon the relation (\ref{eqn}), we will construct two families of surface bundles with fiber genus $g \geq 3$ and $g=2$, respectively. 

\begin{figure}[ht]
 \begin{center}
      \includegraphics[width=17cm]{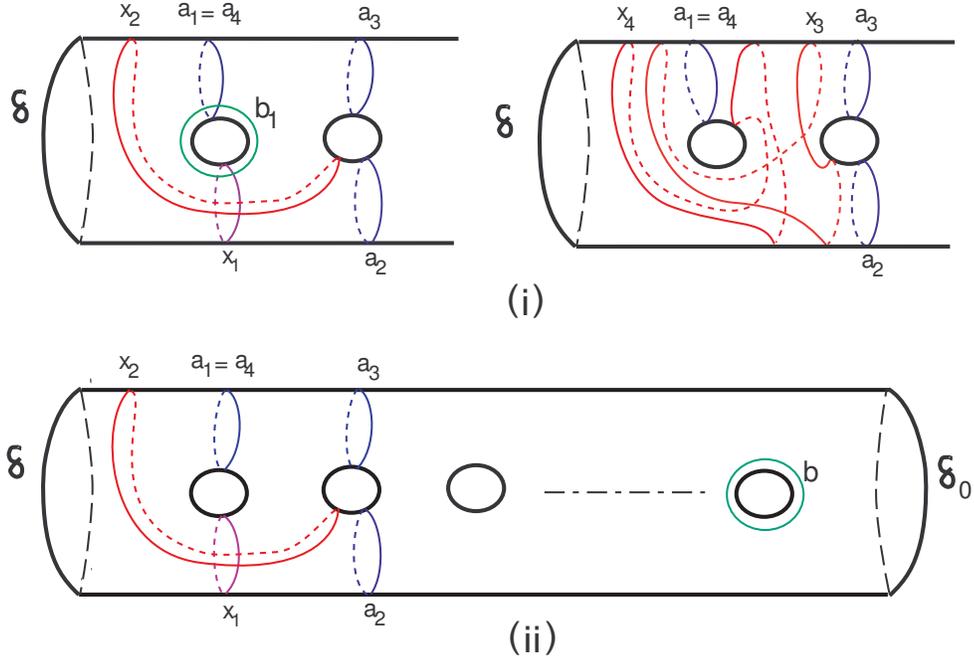}
      \caption{The boundary curves $\delta$ and $\delta_0$ are as given above. The genus two subdomain in $(i)$ embeds into a genus $g \geq 3$ surface as seen in $(ii)$.}
      \label{relationcurves}
   \end{center}
\end{figure}

Let $g \geq 3$ first, and $b$ be a non-separating curve on $\Sigma$ disjoint from $a_i, x_i$, as shown in Figure~\ref{relationcurves}. For $\phi$ and $\psi$ as before (described right after the Equation~\ref{eqn:b} above), the following relation holds in $\Map (\Sigma)$:
\begin{eqnarray}\label{eqn:c}
t_\delta^{2k} 
 &=&\left (\prod_{i=1}^{k}  [t_{x_1}t_{a_2}^{-1}, \phi ]^{t_{x_3}^{i-1}} \right)  [t_{x_3}^k t_{a_4}^{-k}, \psi ] \, 
t_b^m t_b^{-m}   \\
 &=&\left (\prod_{i=1}^{k}  [t_{x_1}t_{a_2}^{-1}, \phi ]^{t_{x_3}^{i-1}} \right)  [t_{x_3}^k t_{a_4}^{-k}, \psi \, t_b^m]   \, ,
\end{eqnarray}
since $\psi$ and $t_b$ commute. 

Capping off the boundary components of $\Sigma$ by two disjoint disks, we will obtain our first family of promised bundles. Under the natural homomorphism from $\Map (\Sigma)$ to $\Map(\Sigma_g)$, the mapping class group of the closed genus $g$ surface, induced by capping off the boundary components of $\Sigma$ by two disjoint disks, the above relation maps to the relation
\begin{eqnarray}\label{eqn:d}
1
 &=&\left( \prod_{i=1}^{k}  [t_{a_1} t_{a_2}^{-1}, 1 ]^{t_{a_3}^{i-1}} \right)
 [t_{a_3}^k t_{a_1}^{-k} , t_b^m]) =  [t_{a_1} t_{a_2}^{-1}, 1 ]^k \, [t_{a_3}^k t_{a_1}^{-k}, t_b^m] \, 
\end{eqnarray}
in $\Map (\Sigma_g)$. Recall from earlier that such a factorization of identity specifies a genus $g$ bundle over a genus $h=k+1$ bundle and the very fact that they lift to relations as above hands us distinguished sections with self-intersections read off from the powers of boundary parallel Dehn twists in these relations. Thus we get a family of genus $g \geq 3$ surface bundles $f_m: X_m \to \Sigma_h$ with two sections $S$ and $S_0$ of self-intersections $2h-2$ and $0$, respectively. 

Now let $g=2$, and $b_1$ be the non-separating curve on $\Sigma$ which is geometrically dual to $a_1$; See Figure \ref{relationcurves}. The following relation holds in $\Map (\Sigma)$:
\begin{eqnarray}\label{eqn:e}
t_\delta^{2k} 
 &=&\left(\prod_{i=1}^{k}  [t_{x_1}t_{a_2}^{-1}, \phi ]^{t_{x_3}^{i-1}} \right)  [t_{x_3}^k t_{a_4}^{-k}, \psi ] \, [t_{b_1}^m , 1]   \, .
\end{eqnarray}
We thus get our second family of bundles: what we have in hand is a family of genus $2$ surface bundles $g_m: Y_m \to \Sigma_{h+1}$, for $h=k+1$, with two sections $S$ and $S_0$ of self-intersections $2h-2$ and $0$, respectively, which are prescribed by:
\begin{eqnarray}\label{eqn:f}
1
 &=&\left( \prod_{i=1}^{k}  [t_{a_1} t_{a_2}^{-1}, 1 ]^{t_{a_3}^{i-1}} \right)
 [t_{a_3}^k t_{a_1}^{-k}] \, [ t_{b_1}^m, 1] =  [t_{a_1} t_{a_2}^{-1}, 1 ]^k \, [t_{a_2}^k t_{a_1}^{-k}] [t_{b_1}^m, 1] \, .
\end{eqnarray}

We will show that the bundles $(X_m, f_m)$ (resp. $(Y_m, g_m)$ together with distinguished sections $S$ yield the desired families. 

Let us first show that the bundles $(X_m, f_m)$ are all flat: We can lift all $t_{a_i}$ and $t_b$ to some $\bar{t}_{a_i}$ and $\bar{t}_b$ in $\text{Diff}^+_0(\Sigma_{g})$ which are curve twists compactly supported in tubular neighborhoods of the associated curves disjoint from each other whenever the pair of curves are disjoint. Thus the supports of $\bar{t}_{a_1}$, $\bar{t}_{a_2}$, $\bar{t}_{a_3}$ are chosen to be pairwise disjoint, and the support of $\bar{t}_b$ is disjoint from the first two. We then have the following lift of the monodromy factorization in $\text{Diff}^+_0(\Sigma_{g})$:
\begin{eqnarray}\label{eqn:g}
\xi
 &=&
[\bar{t}_{a_1} \bar{t}_{a_2}^{-1}, 1 ]^k \, [\bar{t}_{a_3}^k \bar{t}_{a_1}^{-k}, \bar{t}_b^m]
\end{eqnarray}
Since $\bar{t}_{a_1}$ and $\bar{t}_{a_2}$ have supports disjoint from that of $\bar{t}_b$, they commute with it in $\text{Diff}^+_0(\Sigma_{g})$. Hence we get $\xi = 1$ in $\text{Diff}^+_0(\Sigma_{g})$, showing that our bundle is flat. For $(Y_m, g_m)$ we can argue similarly and show that the factorization on the right-hand side of (\ref{eqn:d}) lifts to a factorization of the identity in $\text{Diff}^+_0(\Sigma_{2})$, proving the flatness of these bundles as well.


Next, we verify that for any $m \geq 0$, $(X_m, f_m)$ does not admit a flat connection for which the section $S$ is parallel. This simply follows from Proposition~\ref{labelprop}, since the self-intersection number $2h-2$ of $S$ obstructs this section to be parallel for any $h \geq 2$. As for $(Y_m, g_m)$ along with the distinguished section $S$, recall that if $h$ is the base genus of this bundle, then the self-intersection number of $S$ is $2h-4$. So the Proposition~\ref{labelprop} leads to the same obstruction, provided $h \geq 4$ in this case.

A straightforward calculation using the monodromy factorization of $(X_m, f_m)$ shows that $H_1(X_m)= \Z^{2h+2g-3} \oplus \Z_m $. It follows that $X_m$ and $X_{m'}$ are pairwise non-homotopic for any $m \neq m'$. Similarly, from the monodromy factorization of $(Y_m, g_m)$ we calculate $H_1(Y_m)= \Z^{2h+1} \oplus \Z_m$, leading the same conclusion.
\end{proof}

\begin{rk} \label{fibersums} \rm
The constructions of the above bundles $(X_m, f_m)$ and $(Y_m, g_m)$ have obvious geometric interpretations. These bundles are nothing but \textit{horizontal} and \textit{vertical} stabilizations of the bundle one gets from the very first relation (\ref{eqn:b}) above, as described in Theorem~4 of \cite{B}. That is, they are obtained by \textit{section summing} and \textit{fiber summing} this initial bundle with some standard bundles.  
\end{rk}

We can in fact see that $X_m$ (resp. $Y_m$) do not admit complex structures with either orientation. This follows from Lemma~2 of \cite{B} (also see \cite{Kots}), which in this case states that if $X$ is the total space of a genus $g$ surface bundle over a genus $h$ surface with $g,h \geq 2$, such that $b_2(X)>0$ and $b_1(X)$ odd, then $X$ does not admit a complex structure with either orientation. Since each $X_m$ admits a section, we have $b_2(X_m) \geq 2 >0$. We also see that $b_1(X_m)= 2h + 2g - 3$ is odd. We conclude that $X_m$ does not admit a complex structure with either orientation. The claim is proved for $Y_m$ in an identical way. We conclude that the existence of flat bundles with sections which admit no flat connections making the sections parallel is not relevant to the bundles being holomorphic or not. Product bundles on $\Sigma_h \x \Sigma_h$ with diagonal sections, and in general Atiyah-Kodaira examples covered in \cite{BCS} are all holomorphic, whereas the \linebreak $4$-manifolds $X_m$ and $Y_m$ we obtain above do not admit complex structures with either orientation. We also note that our examples span all possible $g, h \geq 2$ (even when $g=2$ and $h=2,3$, we have one example). Yet another qualitative difference between the bundles used in \cite{BCS} and the ones we constructed above is the existence of the self-intersection zero section $S_0$, which is \emph{disjoint} from the section $S$ of self-intersection $2h-2$. This difference simplifies our derivation of Corollary~\ref{cormorita} greatly:

\begin{proof}[Proof of Corollary~\ref{cormorita}.]
For each $g \geq 3$ and $h \geq 2$, we constructed a genus $g$ surface bundle $f_m: X_m \to \Sigma_h$ with two sections $S$ and $S_0$ of self-intersections $2h-2$ and $0$, respectively. On the other hand for $g=2$ and $h \geq 4$, we obtained a genus $2$ surface bundle $g_m: Y_m \to \Sigma_{h}$ with two sections $S$ and $S_0$ of self-intersections $2h-4$ and $0$, respectively. We will present our arguments for the former, which can be then easily adapted for the latter.

For any $r \geq 2$, let $S_1, \ldots, S_{r-1}$ be sections of $f_m$ obtained by pushing the self-intersection zero section $S_0$ off itself, all \emph{disjoint} from each other and $S$. Let us also label $S$ as $S_{r}$. When $r=1$, we only take $S_1$. (Although the first $r-1$ sections are obviously homotopic, for what follows, what we need is to have \emph{disjoint} sections, not necessarily \emph{distinct} ones.) Thus we have $S_r$ disjoint sections of $(X_m, f_m)$, which prescribes the monodromy map $\mu_{f_m}: \pi_1(\Sigma_{h}) \to \Map (\Sigma)$, where $\Sigma$ is a genus $g$ surface with $r$ marked points. If the subgroup $\mu_{f_m}(\pi_1(\Sigma_{h}))$ of $\Map (\Sigma)$ were to lift to $\Diff ^+ (\Sigma)$, then we would get a lift of the monodromy $\bar{\mu}_{f_m}: \pi_1(\Sigma_{h}) \to \Diff ^+ (\Sigma)$. This however implies that the section $S$ is parallel with respect to some flat connection, which is impossible by Proposition~\ref{labelprop}. Hence $\mu_{f_m}(\pi_1(\Sigma_{h})) < \Map (\Sigma)$ cannot be lifted to $\Diff ^+ (\Sigma)$. Varying $m \in \Z^+$ hands us the infinitely many non-conjugate subgroups, as promised.
\end{proof}




\vspace{0.4in}
\begin{thebibliography}{111}

\bibitem{B} R.~Inanc Baykur, \emph{``Non-holomorphic surface bundles and Lefschetz fibrations,''} to appear in Math. Res. Let.; arxiv.org/abs/1111.3417.

\bibitem{BKM} R.~Inanc Baykur, Mustafa Korkmaz, and Naoyuki Monden; \emph{``Sections of surface bundles \linebreak and Lefschetz fibrations,''} to appear in Trans. Amer. Math. Soc.; arxiv.org/abs/1110.1224.

\bibitem{BCS} Mladen Bestvina, Thomas Church, and Juan Souto; \emph{``Some groups of mapping classes not realized by diffeomorphisms,''} preprint; arxiv.org/abs/0905.2360. 

\bibitem{Bow} Jonathan Bowden; \emph{``On closed leaves of foliations, multisections and stable commutator lengths,''} preprint; http://arxiv.org/abs/1105.4444.

\bibitem{Calegari} Danny Calegari; \emph{``Scl'',} MSJ Memoirs, 20. Mathematical Society of Japan, Tokyo, 2009.

\bibitem{EK} Hisaaki Endo and Dieter Kotschick; \emph{ ``Bounded cohomology and non-uniform perfection of mapping class groups'',} Invent. Math. 144 (2001), no. 1, 169--175.

\bibitem{EMV} Hisaaki Endo, Thomas E. Mark and Jeremy Van Horn-Morris; \emph{``Monodromy substitutions and rational blowdowns,''} J. Topol. 4 (2011), no. 1, 227–-253. 

\bibitem{FM} Benson Farb and Dan Margalit; \textbf{A primer on mapping class groups.} Princeton Mathematical Series, 49. Princeton University Press, Princeton, NJ, 2012. xiv+472 pp. ISBN: 978-0-691-14794-9. 

\bibitem{Kork} Mustafa Korkmaz; \emph{ ``Stable commutator length of a Dehn twist'',} Michigan Math. J. 52 (2004), no. 1, 23--31.

\bibitem{kot} Dieter Kotschick; \emph{``Stable length in stable groups'',} Groups of Diffeomorphisms in honor of Shigeyuki Morita on the occasion of his 60th birthday, Advanced Studies in Pure Mathematics 52, Mathematical Society of Japan 2008, pp. 401--413.

\bibitem{Kots} Dieter Kotschick; \emph{``Signatures, monopoles and mapping class groups'',} Math. Res. Lett. 5 (1998), no. 1-2, 227–-234.

\bibitem{M} John Milnor; \emph{``On the existence of a connection with curvature zero,''} Comment. Math. Helv. 32 (1958), 215–-223. 

\bibitem{Morita} Shigeyuki Morita; \emph{``Characteristic classes of surface bundles,''} Invent. Math. 90 (1987), 551–-577.

\bibitem{Morita2} Shigeyuki Morita; \emph{``Characteristic classes of surface bundles and bounded cohomology.''} A f\^ete of topology, 233–-257, Academic Press, Boston, MA, 1988.

\bibitem{OS}
\textbf{B. Ozbagci}, \textbf{A.\,I. Stipsicz}, \emph{Surgery on contact 3--manifolds and {S}tein surfaces}, Bolyai Society Mathematical Studies 13, Springer, Berlin (2004) {MR}{2114165}

\bibitem{Smith} Ivan Smith; \emph{``Geometric monodromy and the hyperbolic disc,''} Q. J. Math. 52 (2001), no. 2, 217–-228. 

\bibitem{W} John W. Wood; \emph{``Bundles with totally disconnected structure group,''} Comment. Math. Helv. 46 (1971), 257–-273.

\end{thebibliography}
\end{document}